\documentclass[11pt, leqno, a4paper,]{amsart}    

\usepackage[]{mdframed} 
\usepackage{dsfont}
\usepackage{amsfonts,amsmath,amsthm,amssymb,
amscd,latexsym
}  
\usepackage{hyperref} 
\usepackage[normalem]{ulem}
\usepackage{geometry}
\usepackage{todonotes} 
\usepackage{xcolor}
\usepackage{soul}
\usepackage{mathrsfs}
\newcommand{\R}{\mathbb{R}}

\usepackage{nicefrac}

\newtheorem{theorem}{Theorem}[section]
\newtheorem{lemma}[theorem]{Lemma}

\newtheorem{proposition}[theorem]{Proposition}

\theoremstyle{remark}
\newtheorem{remark}[theorem]{Remark}
 
\numberwithin{equation}{section}

\newcommand*{\bydef}{\overset{\rm def}{=}}
\newcommand*{\norm}[1]{\left\Vert #1\right\Vert}
\renewcommand*{\div}{\operatorname{div}}

\newcommand*{\id}{\operatorname{Id}}

\title[Navier--Stokes--Maxwell]{Optimal Time-Decay of global solutions to the Navier--Stokes--Maxwell system}

\author{Haroune Houamed}
\address{
Ko\c c University, Istanbul, Turkey}
\email{\href{mailto:hhouamed@ku.edu.tr}{hhouamed@ku.edu.tr}, \href{mailto:haroune.houamed@nyu.edu}{haroune.houamed@nyu.edu}
}

\author{Slim Ibrahim}
\address{University of Victoria \\
Canada} 
\email{\href{mailto:ibrahims@uvic.ca}{ibrahims@uvic.ca}}

\author{Belkacem Said--Houari}
\address{University of Sharjah \\
United Arab Emirates  } 
\email{\href{mailto:bhouari@sharjah.ac.ae }{bhouari@sharjah.ac.ae }}

\begin{document}
	 
	\begin{abstract}
We prove the existence and uniqueness of global-in-time solutions to the Navier--Stokes--Maxwell (NSM) system for small initial data in the critical Fujita--Kato space $\dot H^{\nicefrac{d}{2}-1}(\mathbb R^d)$, in any dimension $d\geq 3$.   Under the additional assumption that the initial data belong to the Besov space $\dot B^{-\nicefrac{d}{2}}_{2,\infty}(\mathbb R^d)$, we also establish the optimal decay rate $t^{-\nicefrac{s}{2}-\nicefrac{d}{4}}$ at infinity
of the solution in $\dot H^s(\mathbb R^d)$, for any $s\in\left(-\nicefrac{d}{2},\nicefrac{d}{2}-1\right]$. This is achieved by   constructing  a Lyapunov functional that is  equivalent to the pointwise energy  on  the Fourier  side, and by combining it with an adaptation of the Fourier splitting method to establish its optimal decay rate.

All the analysis carried out here -- from the global existence theory to the study of the large-time behavior -- is performed within a framework that is uniform with respect to the speed of light $c\in(0,\infty)$. In particular, in the non-relativistic limit $c\to\infty$, this allows us to recover  the same results   for the corresponding   limiting magnetohydrodynamic   system. 
\end{abstract}

\maketitle

\tableofcontents
\section{Introduction and state of the art}
 
In this paper, we are interested in the long-time behavior of solutions to the        incompressible Navier--Stokes--Maxwell  (NSM) equations
\begin{subequations}\label{Model}
\begin{equation}\label{Main_System}
	\begin{cases}
		\begin{aligned}
			\text{\tiny(Navier--Stokes's equation)}&&&\partial_t u +u \cdot\nabla u   = \nu \Delta u- \nabla p + j \times B, &\div u =0,&
			\\
			\text{\tiny(Amp\`ere's equation)}&&&\frac{1}{c} \partial_t E - \nabla \times B =- j , &
			\\
			\text{\tiny(Faraday's equation)}&&&\frac{1}{c} \partial_t B + \nabla \times E  = 0 , &\div B = 0,   & 
		\end{aligned}
	\end{cases}%\tag{NSM}
\end{equation}
supplemented with either of the following two versions of Ohm's law:
 \begin{equation}\label{Ohm:laws}
 	\begin{aligned}
 		\text{\tiny(general Ohm's law)}&&&j= \sigma \big( cE + u \times B\big),   
 		\\
 		\text{\tiny(solenoidal Ohm's law)}&&&j= \sigma \big( cE + \mathbb P(u \times B)\big), &\div j = \div E = 0,&
 	\end{aligned}
 \end{equation}
  and the initial condition 
 \begin{equation}
(u,E,B)|_{t=0} =  (u_0,E_0,B_0).
\end{equation}
\end{subequations}

 Here, $u$, $E$, $B$, and $j$ denote the fluid velocity, the electric field, the magnetic field, and the electric current density. The parameters $\nu$, $\sigma$, and $c$ are positive and represent the fluid viscosity, the electrical conductivity, and the speed of light, respectively. The system is naturally formulated in three spatial dimensions, where the cross product and curl operators are understood in their standard vectorial sense.

System \eqref{Model} not only inherits the fundamental analytical difficulties associated with the Navier- Stokes equations, but also exhibits additional challenges primarily due the nonlinear coupling of the velocity equation and the electromagnetic system through the Lorenz force $j\times B$ and also due to the hyperbolic nature of the Maxwell equations, which does not provide direct dissipation for all components of the electromagnetic field.

Our main objective is to develop a Fujita--Kato type theory for the NSM system, in the spirit of the classical well-posedness theory for the incompressible Navier--Stokes equations with initial data in the   homogeneous Sobolev space $\dot H^{\nicefrac{d}{2}-1} (\mathbb R^d)$. In particular, we establish existence and uniqueness results, together with a detailed description of the long-time behavior of solutions, in a critical functional framework adapted to the coupled fluid--electromagnetic system.

A central feature of our analysis is that all the estimates are uniform with respect to the speed of light $c$. This uniformity allows us to investigate the non-relativistic limit $c\to \infty$ and, in particular, to recover the corresponding well-posedness and long-time behavior results for the incompressible magnetohydrodynamic (MHD) system:
\begin{equation}\label{MHD} 
	\begin{cases}
		\begin{aligned}
			 &\partial_t u +u \cdot\nabla u   = \nu \Delta u- \nabla \left( p + \frac{1}{2} |B|^2\right) + B \cdot \nabla  B, &\div u =0,& 
			 \\
			 &  \partial_t B + u \cdot \nabla  B - \frac{1}{\sigma}\Delta B = B \cdot \nabla  u , &\div B =0.   & 
		\end{aligned}
	\end{cases}\tag{MHD}
\end{equation}

 Unlike the NSM formulation above, the MHD system does not rely on the three-dimensional cross product and can therefore be formulated in arbitrary spatial dimensions $d\geq 3$. Thus, although the electromagnetic NSM system is intrinsically associated with three-dimensional vector calculus, the functional-analytic framework underlying our analysis extends accordingly to the corresponding $d$-dimensional MHD equations. The precise sense in which the NSM system and the associated nonlinear terms can be understood beyond three spatial dimensions will be discussed later on.

The operator 
  $\mathbb P \bydef \text{Id} - \nabla \Delta^{-1} \div $, appearing the second version of Ohm's law \eqref{Ohm:laws}, denotes   Leray's projector on divergence-free vector fields. The divergence-free condition on the magnetic field is not an additional constraint that renders \eqref{Main_System} overdetermined. Rather, it is propagated by the evolution: taking the divergence of Faraday's equation  yields
\begin{equation*}
	\partial_t \div B=0,
\end{equation*}
so that $\div B=0$ is preserved by the flow whenever it is satisfied initially. Similarly for the solenoidal Ohm's law \eqref{Ohm:laws}, when $\div j=0$, Amp\`ere's equation propagates the condition $\div E=0$.

The Navier--Stokes--Maxwell system describes the evolution of a viscous, electrically conducting fluid coupled to an electromagnetic field. It arises naturally in the mathematical modeling of plasmas, charged gases, and conducting fluids, where the fluid is subject to the electromagnetic Lorentz force $j\times B$.  
We refer the interested reader to \cite{bis-book,D-book} for further details on the physical background of plasma modeling and to \cite{as, JLZ23} for a mathematical derivation of \eqref{Model} through the analysis of viscous incompressible hydrodynamic regimes arising from Vlasov--Maxwell--Boltzmann systems.
 
Before presenting our results, we briefly review the existing mathematical literature on the Navier--Stokes--Maxwell (NSM) system. Over the years, considerable progress has been made in understanding the well-posedness and long-time dynamics of this coupled fluid-electromagnetic model, using a variety of analytical techniques adapted to its mixed parabolic--hyperbolic structure. The following overview highlights the main milestones most relevant to the present work.  

\medskip
\noindent\textbf{The Two-dimensional case.} 
 The plasma model   \eqref{Model} has attracted increasing attention recently, starting with the work of Masmoudi \cite{MN}, who established the global existence and uniqueness of solutions for initial data $(u_0,E_0,B_0)$ in $L^2\times H^s\times H^s(\mathbb R^2)$, for any $s>0$. These solutions also satisfy the energy inequality
\begin{equation}\label{energy:L2}
	\frac{1}{2}\norm {(u,E,B)(t)}_{L^2(\mathbb R^d)}^2 +  \int_{0}^t \|( \sqrt{ \nu} \nabla u, \tfrac{j}{\sqrt{\sigma}} )(\tau)\|_{L^2(\mathbb R^d)}^2d\tau   \leq \frac{1}{2}\norm {(u_0,E_0,B_0)}_{L^2(\mathbb R^d)}^2 .
\end{equation}
At the energy level, this inequality is, to the best of our knowledge, essentially the only  known a priori information that follows directly from the intrinsic structure of the system. 

It is worth noting that establishing the existence of solutions in the Leray framework, namely, by taking $s=0$ in \cite{MN}, has  long been one of the central open problems for the Navier--Stokes--Maxwell system  \eqref{Model}. The main difficulty in adapting the classical compactness arguments developed for the Navier--Stokes equations stems from the lack of compactness of the Lorentz force $j\times B$ when relying solely on the  energy inequality, above.

In an attempt to overcome this difficulty, the authors of \cite{IK2011} established the  global well-posedness for sufficiently small initial data in
\begin{equation*}
\dot B^0_{2,1}\times L^2_{\mathrm{log}}\times L^2_{\mathrm{log}}(\mathbb R^2),
\end{equation*}
where the subscript ``$\mathrm{log}$'' refers to an additional logarithmic condition on the high-frequency components of the initial data. This additional regularity is propagated for positive times, under a suitable smallness condition on the initial data, and provides precisely the compactness needed to overcome the aforementioned difficulty associated with the Lorentz force.
 We point out that the Besov regularity assumption $\dot B^0_{2,1}(\mathbb R^2)$ was subsequently relaxed to the natural energy space $L^2(\mathbb R^2)$ in \cite{GIM14}.

Besides these developments, we refer to \cite{a19, ag20, AHB24} for other recent advances in the analysis of \eqref{Model} and its inhomogeneous version.

\medskip
\noindent\textbf{The Three-dimensional case.}  
When $d=3$, the energy inequality is far from being sufficient to develop a satisfactory well-posedness theory. The available well-posedness and ill-posedness results \cite{FK64} and \cite{X24}, respectively, for the Navier--Stokes equations (corresponding to the case $B=E=0$ in \eqref{Model}) suggest that the homogeneous Sobolev space $\dot H^{\nicefrac{1}{2}}(\mathbb R^3)$ is optimal for the well-posedness theory in the sense of Hadamard in the Sobolev spaces $\dot H^{s}(\mathbb R^3)$.
 In the context of a general small-data theory for \eqref{Model} at this critical regularity level, it was proved in \cite{IK2011} that, for $d=3$, sufficiently small initial data in
\begin{equation*}
\dot B^{\frac{d}{2}-1}_{2,1}\times \dot H^{\frac{d}{2}-1}\times \dot H^{\frac{d}{2}-1}(\mathbb R^d)
\end{equation*}
generate a unique global solution to \eqref{Main_System}. This result was subsequently improved in \cite{GIM14}, where the Besov regularity assumption on the velocity field was relaxed, yielding global well-posedness for sufficiently small initial data in
\begin{equation*}
\dot H^{\frac{d}{2}-1}\times \dot H^{\frac{d}{2}-1}\times \dot H^{\frac{d}{2}-1}(\mathbb R^d),
\end{equation*}
for $d=3$.

\medskip
\noindent\textbf{Asymptotic behaviors.}  
Another important feature of \eqref{Model} lies in its connection with the magnetohydrodynamic (MHD) system. Indeed, in the non-relativistic regime, when the speed of light tends to infinity (or, in physical terms, when the velocity of the conducting fluid is negligible compared with the speed of light), the full Navier--Stokes--Maxwell (NSM) system can be formally approximated by the simpler \eqref{MHD} system whose energy balance  is given by
\begin{equation}\label{Energy_MHD_weak}
\frac{1}{2} \|(u,B)(t)\|_{L^2}^2 
+\int_0^t \|\nabla (\sqrt{\nu} u,\tfrac{B}{\sqrt{\sigma}} )(\tau)\|_{L^2}^2
 \,d\tau
\leq
\frac{1}{2} \|(u_0,B_0)\|_{L^2}^2 .
\end{equation}
Unlike \eqref{energy:L2}, this energy bound is sufficient to produce Leray-type solutions for \eqref{MHD}.

The main difference between \eqref{Model} and \eqref{MHD} is that the former is only partially dissipative and retains a non-negligible dispersive component inherited from the Maxwell equations, whereas the latter is fully dissipative. Roughly speaking, this can formally be seen by comparing the respective energy  estimates \eqref{energy:L2} and \eqref{Energy_MHD_weak}. 

This fundamental difference makes the global analysis of \eqref{Model} considerably more challenging than that of \eqref{MHD}. Nevertheless, starting from the work by Ars\'enio and Gallagher \cite{ag20}, a significant progress has been made in the study of the singular non-relativistic limit $c\to\infty$, both in the viscous setting \cite{ahh24, AHB24, aim15} and in the inviscid setting \cite{ah2, ah, KKL26}. We also refer the interested reader to other related works on the derivation of the MHD system from a two-fluid NSM model \cite{aim15, as, GZ25, JLT20, PWQ22, WYY25, ZZ21}.

The close structural analogy between the velocity and magnetic-field equations in \eqref{MHD} makes it possible to extend several classical results concerning the long-time behavior of global solutions to the Navier--Stokes equations (corresponding to the case $B=0$) to the MHD setting. By contrast, the long-time dynamics of solutions to the full NSM system \eqref{Model} remain largely unexplored. To the best of our knowledge, the only result currently available that provides insight into this question is a consequence of the work of Ars\'enio, Hassainia, and the first author in three space dimensions under axisymmetric assumptions; see \cite[Theorem 1.3]{ahh24}.
 More precisely, it was shown    therein that the global axisymmetric solutions of \eqref{Model} converge uniformly in time, for $t\in\mathbb R_+$, as $c\to\infty$, toward the corresponding global solutions of the MHD system \eqref{MHD}. In particular, that result  shows that, for sufficiently large values of $c$, the long-time dynamics of solutions to \eqref{Model} and \eqref{MHD} remain closely related in the case of three-dimensional axisymmetric solutions.

\medskip
\noindent\textbf{Aim and contribution of the paper.}
In light of the above discussion, and inspired by both classical results \cite{S86} and more recent developments \cite{BS18, D26, HDA20} for related models, we are interested in the long-time behavior of solutions to \eqref{Model}. The main contributions of our analysis can be summarized in two steps.
First, in Theorem \ref{thm:critical:wp}, we establish an improved well-posedness result, valid in any dimension $d\geq 3$, for a generalization of the NSM system to higher dimensions (see the next section for a discussion of the extension of the NSM model to higher dimensions). Our approach relies on the optimal parabolic estimate
\begin{equation*}
\norm{e^{t\Delta}u_0}_{L^2(\mathbb R^+;\dot B^{\frac{d}{2}}_{2,1}(\mathbb R^d))}
\lesssim
\norm{u_0}_{\dot H^{\frac{d}{2}-1}(\mathbb R^d)},
\end{equation*}
recently established by Ars\'enio and the first author in \cite{ah}. This estimate allows us to control the velocity field associated with \eqref{Model} in
$L^2(\mathbb R^+;\dot B^{\nicefrac{d}{2}}_{2,1}(\mathbb R^d))$, rather than in the larger space $L^2(\mathbb R^+;\dot H^{\nicefrac{d}{2}} \cap  L^\infty(\mathbb R^d))$, as was previously obtained in \cite{GIM14}.

Second, and more importantly, we address the long-time behavior of solutions to \eqref{Model} in arbitrary dimensions $d\geq 3$. In Theorem \ref{thm:decay}, we establish the optimal decay rate for global solutions to \eqref{Model} in the framework of Fujita--Kato spaces. Our approach provides a simple way of analyzing the time evolution of global solutions. It is based on the construction of a suitable Lyapunov functional, which yields the desired decay rate in Sobolev spaces through an appropriate adaptation of the Fourier splitting method.

It is worth emphasizing that all the results established in our theorems hold uniformly for $c\gg1$. In particular, this uniformity allows us to recover the corresponding results for the limiting MHD model in the non-relativistic limit $c\to\infty$.

 We defer to Section \ref{sec:main} a more detailed discussion of the refined control of the velocity field, its role in the long-time analysis of \eqref{Model}, and the main ideas underlying our approach of  obtaining the optimal decay rate. 
  
\section{A natural extension to higher dimensions of NSM system}\label{section:extension}

The usual vectorial formulation of the Navier--Stokes--Maxwell system \eqref{Model} is specific to three dimensions because it involves the cross product and the curl operator. We give here an exterior-calculus formulation that makes \eqref{Model} meaningful in every dimension $d\geq 3$ and coincides with the classical system when $d=3$. The velocity and magnetic field are represented by 1-forms, whereas the electric field and current density are represented by 2-forms. Throughout this section, vectors and 1-forms are identified by the Euclidean metric. For a complete introduction and detailed discussion about differential forms used in this section, we refer to chapter 6 from the book \cite{AMR-book88}.

\subsection{Elementary tensor analysis}  For $k\in \{ 1,\dots , d\}$, we denote by $\Lambda^k(\mathbb R^d)$ the space of alternating $k$-tensors $T$ on $\mathbb R^d$, i.e.,  $T: (\mathbb R^d)^k \to \mathbb R$ 
such that 
\begin{equation*}
	T\big ( v_{\sigma(1)}, \dots,  v_{\sigma(k)} \big ) = \text{sign}(\sigma) T ( v_1, \dots,  v_d)  ,
\end{equation*}
for every permutation $\sigma\in S_k$, with $S_k$ being the symmetric group on $k$ elements, where 
\begin{equation*}
	\text{sign}(\sigma)= \begin{cases}
		1, \qquad \text{if } \sigma \text{ is obtained by an even number of swaps},
		\\
		-1 \qquad \text{if } \sigma \text{ is obtained by an odd number of swaps}.
	\end{cases}
\end{equation*}
In particular, 
\begin{equation*}
	\Lambda^2 (\mathbb R^d)= \{A= (A)_{1\leq i,j\leq d} : A_{ij}=- A_{ji} \} 
\end{equation*}
is the space of antisymmetric $2$-tensors. For $A,C\in \Lambda^2 (\mathbb R^d)$, we use the Euclidean inner product
\begin{equation*}
	A:C \bydef \frac{1}{2} \sum_{i,j=1}^d A_{ij} C_{ij},
\end{equation*}
and the associated norm 
\begin{equation*}
	|A|^2 \bydef A:A.
\end{equation*}
\subsection{Exterior and interior products}
For two vectors $a,b\in \mathbb R^d$, we define their exterior product  
$a\wedge b \in \Lambda^2 (\mathbb R^d)$
 by 
\begin{equation}\label{eq:wedge-product}
(a\wedge b)_{ij}
\bydef
a_i b_j-a_jb_i , \qquad 1\leq i,j
 \leq d,
\end{equation}
In particular, $a\wedge b  = -b\wedge a$. In dimension three, after fixing an orientation, the Hodge star map identifies this operation with the usual cross product:
\begin{equation*}
	* (a\wedge b)= (a\times b)^{\flat},
\end{equation*}  
where, for a given vector $v\in \mathbb R^d$, the object $v^{\flat } \in \Lambda(\mathbb R^d)$ represents its 1-form representation.  Accordingly, the operator   $ (a ,b ) \mapsto a\wedge b $   provides a canonical, dimension-independent replacement for the three-dimensional cross product. From now on, for simplicity, we will be using the identification of vectors with 1-forms by omitting the symbol ``$\flat$'', above. 

Unlike the usual cross product, the operator ``$\wedge$'' takes values in $\Lambda^2 (\mathbb R^d) $ instead of $\mathbb R^d$. Given this difference, we shall now consider the use of an interior product, or contraction. More precisely, for $A\in \Lambda^2 (\mathbb R^d)$ and $b\in \mathbb R^d$, we define 
\begin{equation*}
	\iota_b A \in \mathbb R^d
\end{equation*}
by
\begin{equation}\label{eq:interior-product}
(\iota_b A)_i
\bydef -\sum_{j=1}^d A_{ij}b_j.
\end{equation}
With the convention \eqref{eq:wedge-product} for the exterior product and the above normalization of the inner product on 2-forms, one has the algebraic identity
\begin{equation}\label{eq:wedge-contraction}
(\iota_b A)\cdot a=-
A:(a\wedge b),
\end{equation}
which follows from  the elementary computation
\begin{equation*}
	\begin{aligned}
		A:(a\wedge b)
			&= \frac{1}{2} \sum_{i,j=1}^d A_{ij}  (a_i b_j - a_j b_i)
			\\
			&=   \sum_{i,j=1}^d A_{ij}a_ib_j
			\\
			&= -(\iota_b A)\cdot a,
	\end{aligned}
\end{equation*}
where we have used the antisymmetry property of $A$.

Consequently, if $j$ is a $2$-form-valued current density, the natural analogue of the Lorentz force $j\times B $ is 
\begin{equation*}
	\mathbf L _d (j,B)\bydef  \iota _B j.
\end{equation*}
In view of this definition, we have that
\begin{equation}\label{eq:lorentz-cancellation}
\mathbf L_d(j,B)\cdot u= -
j:(u\wedge B).
\end{equation}

\subsection{Generalized curl and its adjoint} We next introduce the differential operators replacing the curl in higher dimensions. Let
\begin{equation*}
	B= \sum_{i=1}^d B_i dx_i
\end{equation*}
be the $1$-form associated with the vector field $B=(B_1,\dots, B_d)\in \mathbb R^d$. We define  the exterior derivative, by setting 
\begin{equation}\label{eq:generalized-curl}
(\textbf{d} B)_{ij} \bydef 
\partial_iB_j-\partial_jB_i, \qquad \text{for all } 1\leq i,j\leq d. 
\end{equation}
Thus, we have that 
\begin{equation*}
	\textbf{d} : \Lambda^1 (\mathbb R^d) \to \Lambda^2 (\mathbb R^d),
\end{equation*}
which will be interpreted as the extension of the curl operator in higher dimensions. The terminology ``curl'' is justified by the fact that, in dimension $d=3$, the Hodge star identifies the $2$-form $\mathbf dB$ with the vector field $\nabla \times B$.
The formal $L^2$-adjoint of $ \textbf{d}$ if then given by 
\begin{equation}\label{eq:generalized-curl-adjoint}
(\textbf{d}^*E)_i
=-\sum_{j=1}^d\partial_jE_{ji} =  \sum_{j=1}^d\partial_jE_{ij},
\end{equation}
for a $2$-form $E$. Hence, the defining adjoint relation gives
\begin{equation}\label{eq:curl-adjoint}
\int_{\mathbb R^d}
(\textbf{d}B):E  = 
\int_{\mathbb R^d}
B\cdot\textbf{d}^*E ,
\end{equation}
for sufficiently regular fields with suitable decay at infinity. 
Notice that  the operator $\textbf{d}$ satisfies the analogue of the classical identity
\begin{equation*}
	\div (\nabla \times B)=0,
\end{equation*}
which is the result of the identity 
\begin{equation*}
	 \mathbf d^2B=0.
\end{equation*}
Moreover, the Hodge Laplacian identity gives 
\begin{equation*}
	(\mathbf d^*\mathbf d +\mathbf d\mathbf d^*) B= - \Delta B,
\end{equation*}
with the Euclidian sign convention used here. Since 
\begin{equation*}
	\mathbf d^* B=- \div B
\end{equation*}
we obtain
\begin{equation}\label{eq:curl-laplacian}
\textbf{d}^*\textbf{d}B=
-\Delta B,
\qquad
\text{whenever }\operatorname{div}B=0.
\end{equation}
Thus, on divergence-free vector fields, the generalized curl retains the same elliptic structure as the classical curl in three dimensions.
\subsection{Reformation of the NSM system in higher dimensions}
We are now in a position to formulate the higher-dimensional version of the NSM system. The unknowns are
\begin{equation*}
	u,B : \mathbb R^+ \times \mathbb R^d \to \mathbb R^d
\end{equation*}
while 
\begin{equation*}
	E,j : \mathbb R^+ \times \mathbb R^d \to \Lambda^2(\mathbb R^d).
\end{equation*}
Accordingly, in any dimension $d\geq 3$, we consider the following replacement of \eqref{Main_System} 
\begin{equation}\label{eq:NSM-higher-dimensional}
\begin{cases}
\begin{aligned}
&&& \partial_tu+u\cdot\nabla u
=
\nu\Delta u-\nabla p+\iota_Bj,
&
\operatorname{div}u&=0,
\\
 &&&\displaystyle \frac1c\partial_tE-\textbf{d} B = -j, &
 \\
 &&& \displaystyle \frac1c\partial_tB+ \textbf{d} ^*E = 0, & \operatorname{div}B&=0,  \end{aligned} \end{cases} \end{equation}
 supplemented with either of the following two versions of Ohm's law:
 \begin{equation}\label{Ohm:laws:d}
 	\begin{aligned}
 		\text{\tiny(general Ohm's law)}&&&j= \sigma \big( cE + u  \wedge B\big),   
 		\\
 		\text{\tiny(solenoidal Ohm's law)}&&&j= \sigma \big( cE + \mathbb P_2(u \wedge B)\big), & \textbf{d}  j = \textbf{d}  E = 0,&
 	\end{aligned}
 \end{equation}
 where 
 \begin{equation*}
 	\mathbb P_2 \bydef \id + \textbf{d}^* \Delta^{-1} \textbf{d} .
 \end{equation*}
 
 The system \eqref{eq:NSM-higher-dimensional}-\eqref{Ohm:laws:d} is the natural differential-form formulation of the NSM equations in higher dimensions $d\geq3$. The first equation is the incompressible Navier--Stokes equation with the generalized Lorentz force $\iota_Bj$. The second and third equations are the generalized Amp\`ere and Faraday equations, respectively, with the curl replaced by the exterior derivative and its formal adjoint. Finally, the generalized Ohm law is obtained by replacing the three-dimensional vector product $u\times B$ by the exterior product $u\wedge B$.

  Let us verify that the basic energy structure is preserved. Taking the \(L^2\) inner products of the first, second, and third equations in \eqref{eq:NSM-higher-dimensional} with $u$, $E$, and $B $, respectively, we obtain 
 \begin{equation}\label{eq:energy-preliminary} \begin{aligned} \frac12\frac{d}{dt}\|u\|_{L^2}^2 +\nu\|\nabla u\|_{L^2}^2 &= \int_{\mathbb R^d} (\iota_Bj)\cdot u dx, 
 \\ \frac1{2c}\frac{d}{dt}\|E\|_{L^2}^2 - \int_{\mathbb R^d} (\mathbf d B):E\,dx &= -\int_{\mathbb R^d}j:E dx, 
 \\ \frac1{2c}\frac{d}{dt}\|B\|_{L^2}^2 + \int_{\mathbb R^d} B\cdot\mathbf d^*E\,dx &= 0. \end{aligned} \end{equation}
 
  By the adjoint relation \eqref{eq:curl-adjoint}, the two Maxwell coupling terms cancel. Furthermore, by \eqref{eq:wedge-contraction}, $ (\iota_Bj)\cdot u = - j:(u\wedge B)$.
  
  On the other hand, by Ohm's law, 
  \begin{equation*}
-j:E dx= - \frac{1}{\sigma c} |j|^2 + \frac{1}{c } j : (u \wedge B).
  \end{equation*}
Combining this identity with the Lorentz contribution in the first equation of \eqref{eq:energy-preliminary} yields the corresponding energy balance. With the normalization of the electromagnetic variables used in \eqref{eq:NSM-higher-dimensional}, one obtains
\begin{equation}\label{eq:higher-dimensional-energy}
\frac12\frac{d}{dt} 
\|(u,E,B)\|_{L^2}^2    +
\nu\|\nabla u\|_{L^2}^2 +
\frac1{\sigma }\|j\|_{L^2}^2= 0.
\end{equation}
 In particular, the essential algebraic cancellation between the Lorentz force and the nonlinear term in Ohm's law is preserved in arbitrary dimensions.

Once again, we emphasize on the relation with the classical formulation in dimension three. For a vector field $e$ corresponding to a $2$-form $E$ via the relation 
 \begin{equation*}
 	E=* e ^\flat,
 \end{equation*}
 it holds that 
 \begin{equation*}
 	\textbf{d} E=  (\div e) dx_1 \wedge dx_2 \wedge  dx_3
 \end{equation*}
 while 
 \begin{equation*}
 	\textbf{d}^*E=(\nabla \times e)^\flat. 
 \end{equation*}
 Accordingly, the divergence-free condition for $2$-forms is given by 
 \begin{equation*}
 	\textbf{d} E = 0,
 \end{equation*}
 which is the basis behind the definition of the solenoidal Ohm's law, above. 
 
 Now, fixing an orientation on $\mathbb R
^3$, the Hodge star operator provides an isometric identification
\begin{equation*}
	* : \Lambda^2 (\mathbb R^3) \to \Lambda ^1 (\mathbb R^3) \sim \mathbb R^3.
\end{equation*}
Under this identification, it holds that 
\begin{equation*}
	* (u\wedge B)\sim  u\times B 
\end{equation*}
and 
\begin{equation*}
	* (\textbf{d} B)\sim \nabla \times B.
\end{equation*}
Likewise, the contraction operation is related to the cross product by
\begin{equation*}
	\iota_B j \sim j\times B,
\end{equation*}
with the orientation and the corresponding Hodge identification fixed. Consequently, after identifying the 2-form-valued fields $E$ and $j$ with their Hodge dual vector fields, \eqref{eq:NSM-higher-dimensional} reduces to the usual three-dimensional NSM formulation.

All in all, this reformulation therefore provides a canonical extension of the NSM system to arbitrary spatial dimensions $d\geq 3$ without introducing a non-canonical $d$-dimensional vector cross product. The velocity and magnetic fields remain $\mathbb R^d$-valued vector fields, while the electric field and current density are naturally represented by antisymmetric $2$-tensor-valued fields. The three-dimensional operations 
\begin{equation}\label{vector:notation:3}
	u\times B, \qquad  j\times B, \qquad \nabla \times B
\end{equation}
 are replaced, respectively, by
\begin{equation}\label{vector:notation:d}
	u \wedge B, \qquad \iota_B j , \qquad \mathbf dB,
\end{equation}
while the remaining curl operation is represented by the formal adjoint $\mathbf d^*$.

\subsubsection*{\textbf{Conventional notation}} For convenience, and in order to keep the presentation of the proofs of our main results as simple as possible, we henceforth use the notation of \eqref{Model} for all spatial dimensions $d\geq3$, with the understanding that, when $d>3$, the system is interpreted through the higher-dimensional reformulation \eqref{eq:NSM-higher-dimensional}. Likewise, we continue to use the vectorial notation introduced in \eqref{vector:notation:3} in arbitrary dimensions $d\geq3$, instead of the more precise notation introduced in \eqref{vector:notation:d}. This abuse of notation is harmless for our purposes, since the algebraic and analytic properties of these objects that are used throughout our arguments are identical in the two formulations.

\section{Statement and discussion of the main results}\label{sec:main}

In the present work, we aim to develop a general framework for the Navier--Stokes--Maxwell system \eqref{Model} that allows us to investigate the long-time behavior of its solutions. This constitutes a first step toward a more refined analysis of questions such as asymptotic stability and the characterization of asymptotic profiles as $t\to\infty$.

Taking the Navier--Stokes equations as our guiding model, we formulate our analysis within the framework of Fujita--Kato solutions \cite{Fujita_1964}. Our first result is stated in this setting as follows:

\begin{theorem}\label{thm:critical:wp}
	Let $d\geq 3$.  
	There exists a sufficiently small $\delta>0$ such that for any initial data 
	 $(u_0, E_0, B_0)\in (\dot H^{\frac{d}{2}-1} (\mathbb R^d))^3$   satisfying the condition 
	\begin{equation*}
		\norm {(u_0,E_0,B_0)}_{\dot H^{\frac{d}{2}-1} (\mathbb {R}^d)} < \delta,
	\end{equation*}
	there exists a  unique global-in-time solution to the Navier--Stokes--Maxwell system \eqref{Model} enjoying the bound 
	\begin{equation*}
		\norm {(u,E,B)}_{C(\mathbb {R}^+; \dot H^{\frac{d}{2} -1})} + \norm {cE}_{L^2 (\mathbb {R}^+; \dot H^{\frac{d}{2} -1})} +  \norm {u}_{L^2 (\mathbb {R}^+; \dot B^{\frac{d}{2} }_{2,1})} \leq 2 \norm {(u_0,E_0,B_0)}_{\dot H^{\frac{d}{2}-1}  }.
	\end{equation*}

\end{theorem}

Before  discussing our second result on the long time behavior, let us highlight some remarks about Theorem \ref{thm:critical:wp}.
 
\begin{remark}
A result similar to the preceding theorem was established in \cite{GIM14} in dimension $d=3$, under the same assumptions on the initial data. Theorem \ref{thm:critical:wp} above differs in that it provides the refined control $ L^2 (\mathbb {R}^+; \dot B^{\nicefrac{d}{2} }_{2,1} (\mathbb R^d))$ instead of the slightly weaker control in $ L^2 (\mathbb {R}^+; \dot H^{\nicefrac{d}{2} } \cap L^\infty (\mathbb R^d))$.
	This refined estimate is not only instrumental in closing the global estimates at the   Sobolev regularity level, but also plays a crucial role in propagating a lower-order regularity, more precisely the Besov regularity $\dot B^{-\nicefrac{d}{2}}_{2,\infty}(\mathbb R^d)$, which is essential for establishing the optimal decay rate of the global solutions. The precise propagation result for this Besov regularity is stated in Proposition \ref{Prop_negative_regularity}, below.
	\end{remark}

\begin{remark}
Notice that the solutions constructed in the first part of the previous theorem do not necessarily have finite energy. Moreover, the smallness condition imposed on the initial data depends only on the size of the critical norm $\dot H^{\nicefrac{d}{2}-1}$, which is critical with respect to the scaling of the velocity field and is consistent with the corresponding critical well-posedness theory for the Navier--Stokes equations in dimension $d$; see, for instance, \cite{Fujita_1964, X24}.
\end{remark}

\begin{remark}
It is readily seen that, provided the initial data remain uniformly bounded with respect to $c\in(0,\infty)$, the solutions constructed in Theorem \ref{thm:critical:wp} are also uniformly bounded with respect to $c$. This uniform control provides a natural framework for recovering the MHD system and its Fujita--Kato solutions in the non-relativistic limit $c\to\infty$. Moreover, it opens the way to studying strong convergence within the Fujita--Kato framework, building, for instance, on recent asymptotic analyses of the Euler--Maxwell and Navier--Stokes--Maxwell systems, such as those developed in \cite{ahh24, ah2, KKL26}.
\end{remark}

\begin{remark}
%%%%%%%%%%%
A recent work by Coiculescu and Palasek \cite{Coiculescu:2026aa} shows that the perturbative approach to the Navier--Stokes equations cannot, in general, be extended to arbitrarily large initial data in the critical scaling-invariant space $BMO^{-1}$. More precisely, they showed that for large initial data in $BMO^{-1}$  the Navier--Stokes equations admit multiple global solutions, thereby demonstrating the failure of uniqueness in this setting. Consequently, the smallness assumption in the celebrated Koch--Tataru theory \cite{KT01}   is not merely a technical assumption of the fixed-point argument, but is in fact essential for the perturbative global well-posedness theory in $BMO^{-1}$.
 Although a global well-posdeness phenomenon for large data  remains completely open in the Fujita--Kato critical space $\dot{H}^{\nicefrac{1}{2}}$, the possibility that such a result may fail cannot presently be ruled out. From this perspective, the smallness assumption in Theorem~\ref{thm:critical:wp} is somehow justified and appears to be a natural requirement for obtaining global well-posedness by perturbative methods.
%%%%%%%%%% 
\end{remark}  

With Theorem \ref{thm:critical:wp} at hand, a natural question is whether the global solutions constructed therein ``relax'' toward an equilibrium, which in the present setting is the zero solution, and whether this large-time asymptotic behavior occurs at an optimal rate. We answer both questions affirmatively in our second main result, stated below.
   \begin{theorem}\label{thm:decay}
   	Under the assumptions of Theorem \ref{thm:critical:wp}, supposing additionally that 
   	\begin{equation*}
   		(u_0,E_0,B_0)\in \dot B^{-\frac{d}{2}}_{2,\infty}(\mathbb R^d),
   	\end{equation*}
   	then the global solution of \eqref{Model} enjoys the bound
   	\begin{equation*}
   		\norm {(u,E,B)(t)}_{\dot H^s (\mathbb R^d)}  \leq C_0 (1+t)^{-\frac{s}{2} -\frac{d}{4}},
   	\end{equation*}
   	for all $t\geq 0$ and  $s\in (-\frac{d}{2},\frac{d}{2}-1]$, where $C_0>0$ depends only on the $\dot B^{-\frac{d}{2}}_{2,\infty} \cap \dot H^{\frac{d}{2}-1}$ norm of the initial data, uniformly with respect to the speed of light $c\in (0,\infty)$. 
   \end{theorem}

\begin{remark} In view of the classical sharp decay estimate \cite{BS18,S86}
\begin{equation*}
	\norm {e^{t\Delta} f_0}_{\dot H^{s}} \leq C(\norm {f_0}_{\dot{{H}} ^{s}\cap L^1 (\mathbb{R}^d)}) (1+t)^{- \frac{s}{2} -\frac{d}{4}}, \qquad t\geq 0,
\end{equation*}
 the decay rate of the global solutions obtained for the global solutions by our theorems is optimal. 
 Moreover, we do not require the initial data to be in $L^1(\R^d)$, which is replaced here by the weaker assumption $ u\in  \dot{{B}}_{2,\infty}^{-\nicefrac{d}{2}}(\mathbb{R}^d)$.  We point out that the use of  initial data in $\dot{{B}}_{2,\infty}^{-\nicefrac{d}{2} }(\mathbb{R}^d)$ or more generally in the spaces $\dot{{B}}_{2,\infty}^{-\eta }(\mathbb{R}^d) $, for $\eta \in (0, \nicefrac{d}{2}]$, was first introduced by Sohinger and Strain \cite{Sohinger_Strain_2014} in the context of the Boltzmann equation, and by Xu and Kawashima \cite{Xu_Kawashima_2015} for partially dissipative hyperbolic systems. We also emphasize that, in contrast to the decay result of \cite{Xu_Kawashima_2015}, which relies on a smallness assumption on the initial data in $B^{-\nicefrac{d}{2}}_{2,\infty}(\R^d)$; our large-time behavior result established in Theorem \ref{thm:decay} does not require any smallness assumption on the initial data in   $B^{-\nicefrac{d}{2}}_{2,\infty}(\R^d)$, nor in any other subcritical space. 
\end{remark}

\medskip

\noindent\textbf{Strategy of proof and limitations.}
The proof of the global existence result, namely Theorem \ref{thm:critical:wp}, is based on a compactness  
argument combined with a priori estimates that exploit the optimal parabolic smoothing effect for the velocity field. More precisely, starting from initial velocity data $u_0\in\dot H^{\nicefrac d2-1}(\mathbb R^d)$, we establish the bound
\begin{equation*}
	u \in L^2(\mathbb R^+; \dot B^\frac{d}{2}_{2,1} (\mathbb R^d))
\end{equation*}  rather than the weaker control
\begin{equation*}
	u \in L^2(\mathbb R^+; \dot H^\frac{d}{2}\cap L^\infty (\mathbb R^d)).
\end{equation*} 
This stronger bound is crucial for propagating a broad class of natural regularity assumptions on the initial electromagnetic field. It also plays an essential role in controlling the nonlinear terms arising in the analysis of the large-time behavior.
 
 The proof of Theorem \ref{thm:decay}, on the other hand, hinges on constructing a new yet simple Lyapunov functional, defined by 
 \begin{equation*}
 	\frac{L_c(t,\xi)}{ 1+ |\xi|^2 } \bydef \frac{\gamma}{2}  \left |(\widehat u, \widehat E, \widehat B)(t,\xi)\right |^2+ \frac{1}{1+|\xi|^2}F_c(t,\xi) ,
 \end{equation*} 
where $\xi$ denotes the Fourier variable, $\gamma\gg 1$ and  the correction term $F_c$ is given by
 \begin{equation*}
 	F_c(t,\xi) \bydef -\frac{1}{c} \mathrm{Re}\,\langle i\xi\times \widehat{E},\widehat{B}\rangle .
 \end{equation*}
 We do not require any advanced spectral analysis of the propagator corresponding to the linear heat-Maxwell semi-group. Instead, we carry out the large-time analysis using the above Lyapunov functional. By choosing the parameter $\gamma\gg1$ appropriately, this functional captures, up to pointwise equivalence, the main information contained in the energy of the Fourier components of the solution.
 More precisely, we show that the averaged Lyapunov functionals
 \begin{equation*}
	\mathcal{L}_c(t) \bydef \int_{\mathbb R^d} \left ( \frac{L_c(t,\xi)}{ 1+ |\xi|^2 } \right) |\xi|^{d-2}d\xi , \qquad\mathcal{D}_c (t) \bydef \int_{\mathbb R^d} \left ( \frac{L_c(t,\xi)}{ 1+ |\xi|^2 } \right) \left( \frac{|\xi|^2}{1+|\xi|^2} \right) |\xi|^{d-2}d\xi ,
\end{equation*} 
satisfy the differential inequality 
\begin{equation*}
 \begin{aligned}
 	\frac{d}{dt} \mathcal{L}_c(t) + b_0\mathcal{D}_c(t) 
 	&\lesssim   a(t)  \mathcal{L}_c(t),
 \end{aligned}
\end{equation*}
for some constant $b_0>0$ and a suitable function $a\in L^1(\mathbb R^+)$ whose norm depends only on the size of the initial data in $\dot B^{-\nicefrac{d}{2}}_{2,\infty}\cap \dot H^{\nicefrac{d}{2}-1}(\mathbb R^d)$. Since
\begin{equation*}
	\mathcal{L}_c(t) \sim \norm {(u,E,B)(t)}_{\dot H^{\frac{d}{2}-1}(\mathbb R^d)} \qquad \text{ and } \qquad 
	\mathcal{D}_c(t) \sim \norm { \mathcal A(u,E,B)(t)}_{\dot H^{\frac{d}{2}-1}(\mathbb R^d)},
\end{equation*}
where  
\begin{equation*}
	\mathcal A \bydef \sqrt{-\Delta (\id - \Delta)^{-1}},
\end{equation*}
we adapt the core ideas of the Fourier splitting method, originally introduced by M. Schonbek \cite{S86} (see also \cite{BS18}), to establish the optimal decay rate from the above differential inequality. Finally, the pointwise equivalence between the Lyapunov functional and the squared modulus of the Fourier components of the solution allows us to transfer this decay estimates back to the Sobolev norms of the solution. 

The only limitation of our unified proof strategy, which applies uniformly in all dimensions $d\geq 3$ and for all values of $c\in(0,\infty)$, formally including the limiting case $c=\infty$, is that the product laws underlying our analysis break down in dimension $d=2$. Notice that, when $d=2$, the Fujita--Kato theory based on the critical space $\dot H^{\nicefrac d2-1}(\mathbb R^d)$ for the initial velocity coincides with the Leray theory for the Navier--Stokes equations, which is based on $L^2(\mathbb R^d)$ initial data. Our analysis therefore does not cover the two-dimensional case, in line with the longstanding open problem of establishing Leray-type global weak solutions for the Navier--Stokes--Maxwell system  \eqref{Model}.

 \section{Preliminaries}\label{Section_Prel} 

\subsection{Notation}
Throughout the paper, the letter $C$ denotes a generic positive constant that is independent of time and the speed of light and may vary from one occurrence to another. We write $f\lesssim g$ (resp. $f\gtrsim g$) whenever there exists a constant $C>0$, independent of the parameters of interest, such that $f\leq Cg$ (resp. $f\geq Cg$).

We often omit the spatial and temporal domains when writing norms. For instance, $\|\cdot\|_{L^p_T L^q}$ denotes the norm in $L^p(0,T;L^q(\R^d))$. We generally retain the subscript $T$ to emphasize that the Lebesgue time-integration/supremum is performed over the interval $[0,T)$.

For two vector fields $v$ and $w$ in $\mathbb C^d$, with respective components $v_1,\dots v_d$ and $w_1,\dots ,w_d$, we use the notation 
\begin{equation*}
	\left\langle v, w\right\rangle \bydef \left ( \sum _{i=1}^d v_i \overline w _i \right )^\frac{1}{2}
\end{equation*}
for the scalar product of $v$ by $w$.  Accordingly, we define the modulus of all vector fields used throughout the paper.

\subsection{Functional framework}We introduce in this section the functional spaces used throughout the paper. We also recall some product laws that will be repeatedly used in our analysis.

Our analysis is conducted in the framework of homogeneous Sobolev and Besov spaces, and use the standard definitions of these function spaces. In particular, for $s\in\mathbb R$ and $1\leq p,r\leq\infty$, we define
$\dot H^s(\R^d)=\dot B^s_{2,2}(\R^d)$ and equip $\dot B^s_{p,r}(\R^d)$ with the  (semi-)norm
\begin{equation*}
	\|f\|_{\dot B^s_{p,r}(\mathbb R^d)}\bydef\|2^{js} \dot\Delta_jf \|_{\ell^r(\mathbb Z; L^p(\mathbb R^d))},
\end{equation*}
where $(\dot\Delta_j)_{j\in\mathbb Z}$ denotes the homogeneous Littlewood--Paley decomposition. We refer the reader to \cite{bcd11} for the precise definitions and standard properties of these spaces.

 Throughout the paper, we will repeatedly use different versions of product laws to estimate the nonlinear terms arising in the equations under consideration. In the following lemma, we collect most of the product laws that will be needed in our analysis. We omit the details of their proofs and refer instead to \cite{bcd11} for complete references and more detailed discussions.

\begin{lemma}\label{lemma:product} In any dimension $d\geq 1$, for any smooth functions $f$ and $g$, it holds that 
	\begin{equation*}
		\norm {fg}_{\dot B^{s+\eta -\frac{d}{2}}_{2,1} (\mathbb R^d)} \lesssim \norm {f}_{\dot H^{s }(\mathbb R^d) }   \norm {g}_{\dot H^{\eta }(\mathbb R^d) }  ,
	\end{equation*}
	for any real parameters $s,\eta $ with  $s+\eta>0$, $ s< \frac{d}{2}$ and $ \eta< \frac{d}{2}$. 
	
	Moreover, we have, in the case $s=\frac{d}{2}$, that  
		\begin{equation*}
		\norm {fg}_{\dot B^{\eta  }_{2,q} (\mathbb R^d)} \lesssim \norm {f}_{ \dot B^{\frac{d}{2} }_{2,1}(\mathbb R^d) }   \norm {g}_{\dot B^{\eta  }_{2,q}(\mathbb R^d) }  ,
	\end{equation*} 
	for any $\eta \in (-\frac{d}{2}, \frac{d}{2})$, and any $q\in [1,\infty]$.
	
	Finally, for the case $s+\eta=0 $ in the above, one has that 
	\begin{equation*}
		\norm {fg}_{\dot B^{-\frac{d}{2} }_{2,\infty} (\mathbb R^d)} \lesssim \norm {f}_{ \dot B^{s }_{2,p}(\mathbb R^d) }   \norm {g}_{\dot B^{-s  }_{2,p'}(\mathbb R^d) }  ,
	\end{equation*} 
	for any $s<\frac{d}{2}$ and $p\in [1,\infty]$, where $p'$ is the conjugate exponent of $p$. 
 \end{lemma}
 
 Once again, the proof of all the estimates from Lemma \ref{lemma:product} can be done by an appropriate direct application of homogeneous paradifferential calculus. More precisely, we refer the interested reader to Theorems 2.47 and 2.52 from \cite{bcd11} where various estimates for the low-high  and high-high frequency components of the product $fg$ are discussed with details.
 
 \subsection{Sharp parabolic estimates}
 
For convenience, we recall in this section the sharp parabolic  estimates for the heat equation
\begin{equation}\label{heat_Equation}
\partial_t w - \Delta w = f, \quad w|_{t=0} = w_0,
\end{equation}
for a given initial datum $w_0$ and source term $f$.   We begin with the classical estimate in various Besov spaces. 
\begin{lemma}\cite{ag20}\label{Lemma_parabolic_Regularity}
The solution  $w$ of the heat equation \eqref{heat_Equation} enjoys the bound
\begin{equation}\label{Parabolic_Estimate}
\|w\|_{L^m([0,T); \dot{B}^{\sigma+2+\frac{2}{m}}_{p,q}(\mathbb R^d))} \lesssim  \|w_0\|_{\dot{B}^{\sigma+2}_{p,q}(\mathbb R^d)} + \|f\|_{L^r([0,T); \dot{B}^{\sigma+\frac{2}{r}}_{p,q} (\mathbb R^d))}, 
\end{equation}
for any $T\geq 0$, $\sigma\in \R$ and $p,q, r, m\in [1,\infty]$ with  $r\leq  m $ and   $r\leq q \leq m$.  
\end{lemma}

We now state a sharper version of the preceding lemma in terms of the improved possible choice of the third index from the definition of the involved Besov spaces. 
\begin{lemma}\cite[Proposition 2.2 and Proposition 2.5]{ah}\label{Lemma_parab_2}
In the notations of Lemma \ref{Parabolic_Estimate}, the following hold:
\begin{enumerate}
\item If $f=0$, then the $w$ solution of \eqref{heat_Equation} satisfies 
\begin{equation*} 
\|w\|_{L^m(\mathbb {R}^+; \dot{B}^{\sigma+\frac{2}{m}}_{p,1}(\mathbb R^d))} \lesssim  \|w_0\|_{\dot{B}^{\sigma}_{p,m}(\mathbb R^d)} ,
\end{equation*}
as soon as  $m<\infty$.
\item If $w_0=0$, then the maximal regularity estimate  
\begin{equation*} 
\left\|  w \right\|_{L^m([0,T); \dot{B}^{\sigma_1+2}_{p,q}(\mathbb R^d))}
\lesssim  \|f\|_{L^m([0,T); \dot{B}^{\sigma}_{p,q}(\mathbb R^d))},
\end{equation*} 
holds for any $T\geq 0$, $\sigma\in \mathbb R$,  and  $1 \leq p, q \leq \infty$.
\end{enumerate}

\end{lemma}

%%%%%%%%%%%%%

 \section{Global solutions in  Fujita and Kato setting}\label{sec:Fujita:Kato}
	This section is devoted to the proof of Theorem  \ref{thm:critical:wp} on the global existence and uniqueness of solutions to \eqref{Model}. To that end, we reproduce the equation for the    momentum  in Duhamel's formulation, i.e.,
\begin{equation}\label{u:Duhamel}
	u(t) = e^{t\Delta}u_0 - \int _0^t e^{(t-\tau)\Delta} \mathbb {P} \left ( \div (u\otimes u) - c E\times B - u\times B \times B\right)(\tau) d\tau,
\end{equation}	 
where we exploited Ohm's law through \eqref{Ohm:laws} to expand the current density $j$ in the Lorentz force. Note in passing that we choose to present our proofs with the compressible version of Ohm's law, but we emphasize that all the analysis conducted hereafter works for the incompressible version as well.

 We also agree to take the electrical conductivity $\sigma=1$ for the mere sake of simplicity, and we further introduce the spaces 
\begin{equation*}
\begin{aligned}
	&X_t \bydef   L^\infty \big ([0,t);\dot H^{\frac{d}{2}-1}(\mathbb {R}^d) \big ) \cap   L^2 \big ([0,t);\dot B^{\frac{d}{2}}_{2,1}(\mathbb {R}^d) \big ) ,\\
	&Y_t \bydef   L^\infty \big ([0,t);\dot H^{\frac{d}{2}-1}(\mathbb {R}^d) \big ),\\
	&Z_t \bydef   L^2 \big ([0,t);\dot H^{\frac{d}{2}-1}(\mathbb {R}^d) \big ),
	\end{aligned}
\end{equation*}
for all $t\geq 0$, with the convention 
\begin{equation*}
	X_0 \bydef \dot H^{\frac{d}{2}-1}(\mathbb {R}^d) .
\end{equation*}

We are now in a position to prove Theorem \ref{thm:critical:wp}.
\begin{proof}[Proof of Theorem \ref{thm:critical:wp}]
We claim the following estimates for the solution of \eqref{Model} 
\begin{equation}\label{Nonl_Estimate_Main}
	\norm {u}_{X_t} \lesssim \norm {u_0}_{\dot H^{\frac{d}{2}-1}} + \norm {u}_{X_t}^2 +\norm {cE}_{Z_t} \norm {B}_{Y_t} + \norm {u}_{X_t} \norm {B}_{Y_t}^2
\end{equation}
and
\begin{equation} \label{EB:Estimate}
	\norm {(E,B)}_{Y_t} + \norm {cE}_{Z_t} \lesssim \norm {(E_0,B_0)}_{\dot H^{\frac{d}{2}-1}}+ \norm {u}_{X_t} \norm {(E,B)}_{Y_t},
\end{equation}
for all $t\geq 0$. 

For the velocity field, we obtain, by applying the parabolic estimates from Lemmas \ref{Lemma_parabolic_Regularity}  and \ref{Lemma_parab_2} to \eqref{u:Duhamel}, that    
\begin{equation}\label{First_term_X_t}
\|u\|_{X_t}\lesssim \norm {u_0}_{\dot H^{\frac{d}{2}-1}}+\| \div (u\otimes u) + c E\times B + u\times B \times B\|_{L^2_t \dot{B}^{\frac{d}{2}-2}_{2,1}}. 
\end{equation}
 Then, the  product law estimate  from Lemma \ref{lemma:product} yields that 
 \begin{equation}\label{Esti_E_B}
 \begin{aligned}
\|cE\times B\|_{L^2_t\dot{B}^{\frac{d}{2}-2}_{2,1}}\lesssim &\,\|cE\|_{L^2_t\dot{H}^{\frac{d}{2}-1}}\|B\|_{L^\infty_t\dot{H}^{\frac{d}{2}-1}}\lesssim \norm {cE}_{Z_t} \norm {B}_{Y_t}. 
\end{aligned}
\end{equation}
Similarly, by an application of the same lemma, we establish the control 
\begin{equation}\label{B_term_1}
\begin{aligned}
\|  u\times B \times B\|_{L^2_t\dot{B}^{\frac{d}{2}-2}_{2,1}}\lesssim &\, \|u\|_{L^2_t\dot{B}_{2,1}^{\frac{d}{2}}}\|B\times B\|_{L^\infty_t\dot{B}^{\frac{d}{2}-2}_{2,1}}\\
\lesssim&\,\|u\|_{L^2_t\dot{B}_{2,1}^{\frac{d}{2}}}\|B\|_{L^\infty_t\dot{H}^{\frac{d}{2}-1}}^2 \lesssim  \norm {u}_{X_t} \norm {B}_{Y_t}^2,
\end{aligned}
\end{equation}
and
\begin{equation}\label{prodcut_Est_u}
	\| \div (u\otimes u)\|_{L^2_t\dot{B}^{\frac{d}{2}-2}_{2,1}}\lesssim \norm {  (u\otimes u)}_{L^2_{t}\dot B^{ \frac{d}{2} -1}_{2,1}} \lesssim \norm u_{L^2_{t}\dot B^{ \frac{d}{2} }_{2,1}} \norm u_{ L^\infty _{t}\dot B^{ \frac{d}{2} -1}_{2,1}} \lesssim \norm u_{X_t}^2.
\end{equation}
  Therefore, employing  these three inequalities into \eqref{First_term_X_t}  to control the right-hand side concludes    the proof of \eqref{Nonl_Estimate_Main}.
   
 As for the control of the electromagnetic field, employing at first standard energy estimate in Sobolev spaces (see, for instance, \cite[Proposition 2.1]{ag20} and \cite[Lemma 2.5]{HH25}
where most of the arguments of the proof are given in details) we obtain that 
 \begin{equation*}
 	\norm { (E,B)}_{L^\infty _t \dot H^{\frac{d}{2}-1}} + \norm {cE}_{L^2_t \dot H^{\frac{d}{2}-1}} \lesssim \norm { (E_0,B_0)}_{ \dot H^{\frac{d}{2}-1}} + \norm {u\times B}_{L^2_t \dot H^{\frac{d}{2}-1}}.
 \end{equation*}  
The bound for the nonlinear term can easily be derived from a straightforward application of product law estimates  from Lemma \ref{lemma:product} to end up with 
\begin{equation*}
	\norm {u\times B}_{L^2_t\dot H^{\frac{d}{2}-1}}\lesssim \norm u_{L^2_t\dot B_{2,1}^{\frac{d}{2}}}\norm { B}_{L^\infty _t \dot H^{\frac{d}{2}-1}},
\end{equation*}
whereby \eqref{EB:Estimate} follows.

Now we show how to deduce the global control of the solution based on the estimates  \eqref{Nonl_Estimate_Main} and \eqref{EB:Estimate}. To that end,  introducing the time-dependent positive function 
\begin{equation*}
	f(t) \bydef \norm {u}_{X_t} + \norm {(E,B)}_{Y_t} + 
	\norm {cE}_{Z_t}, \quad \text{for all } t\geq 0,
\end{equation*}
we recast the combined bounds on $u$, $E$ and $B$ as 
\begin{equation*} 
	f(t) \leq f_0 + \mathcal G \big (f(t)\big ), \quad \text{for all } t\geq 0,
\end{equation*}
where the function $\mathcal G$ is defined on $[0,\infty)$ by  $r\mapsto \mathcal G(r) \bydef C_1 r^2 + C_2 r^3,$
for some fixed constants $C_1,C_2>0$ which do not depend on the speed of light $c$. Therefore, one can easily deduces the desired global-in-time bound on the function $f$ under suitable assumptions on the datum $f_0$. Indeed, an admissible condition, for instance, would be $\mathcal {G}(2f_0) < f_0$ (see Lemma 5.2 from \cite{ahh24}) 
which can be recast as  $4 C_1 f_0 + 8 C_2 f_0^2 <1.$

This condition can be reformulated (in a weaker form) as $f_0 < \delta,$
for some $\delta>0$ depending only on the universal constants $C_1$ and $C_2$. All in all, under this last smallness condition on the initial data, we deduce that  $f(t)\leq 2 f_0,$
for all $t\geq 0$, thereby completing the proof of the a priori estimates of the solution in terms of the initial data.

With the global control at hand, it is now straightforward to construct a solution to \eqref{Model} that enjoys the global bound stated in Theorem \ref{thm:critical:wp}. Indeed, this can be done by a classical compactness argument applied to a mollified,  or an iterative, version of \eqref{Model}; see for instance \cite{ag20, ahh24, GIM14}.

As for the uniqueness of the solution $(u,E,B)$, we emphasize  that it can be established along the same lines of proof, detailed above for the a priori estimate, in the space $X_t \times Y_t \times  Y_t$, under the smallness condition on the initial data. This completes the proof of the theorem
\end{proof}

Given the control of the velocity field from Theorem  \ref{thm:critical:wp} (in critical spaces in regard to the scaling of the Navier--Stokes equations), it is then possible to propagate any other reasonable regularity for the initial data (for both velocity and electromagnetic field). For a later use in the analysis of the long-time behavior of the global solution constructed in Theorem \ref{thm:critical:wp}, the propagation of the regularity of the initial data is the endpoint Besov space induced by the embedding 
\begin{equation*}
	L^1(\mathbb R^d) \hookrightarrow \dot B^{-\frac{d}{2}}_{2,\infty} (\mathbb R^d),
\end{equation*}
will be crucial to achieve the optimal decay at $t\to \infty$. We thus complete this section by proving the following proposition.
\begin{proposition}\label{Prop_negative_regularity}
	In addition to the assumptions from Theorem \ref{thm:critical:wp}, if 
\begin{equation*}
 (u,E_0,B_0)\in \dot B^{-\frac{d}{2}}_{2,\infty} (\mathbb R^d),
	\end{equation*}
	then the solution of \eqref{Model} also enjoys the bounds 
	\begin{equation*}
		u\in C (\mathbb R^+; B^{-\frac{d}{2}}_{2,\infty}  (\mathbb R^d)) \cap L^2 (\mathbb R^+; B^{-\frac{d}{2}+1}_{2,\infty}   (\mathbb R^d)),
	\end{equation*}
	and 
	\begin{equation*}
		(E,B) \in C (\mathbb R^+; B^{-\frac{d}{2}}_{2,\infty}   (\mathbb R^d)),  \qquad 
  cE \in L^2 (\mathbb R^+; B^{-\frac{d}{2}}_{2,\infty}   (\mathbb R^d)),
	\end{equation*}
	 uniformly with respect to the speed of light $c\in (0,\infty)$.
\end{proposition}

\begin{proof} 
Estimating the velocity based on its  Duhamel reformulation 
 \begin{equation*} 
	u(t) = e^{t\Delta}u_0 - \int _0^t e^{(t-\tau)\Delta} \mathbb {P} \left ( \div (u\otimes u) -j \times B\right)(\tau) d\tau,
\end{equation*}	
we find that 
\begin{equation}\label{negative:es:u}
	\begin{aligned}
		\norm {u}_{L^\infty_t \dot  B^{-\frac{d}{2}}_{2,\infty}} +\norm{u}_{\widetilde L^2_t \dot B^{-\frac{d}{2}+ 1}_{2,\infty}} 
		&\lesssim \norm {u_0}_{  \dot B^{-\frac{d}{2} }_{2,\infty}} + \norm {u\cdot \nabla u}_{\widetilde L^2_t \dot B^{-\frac{d}{2}-1}_{2,\infty}}  +  \norm {J\times B}_{L^1_t  \dot B^{-\frac{d}{2}}_{2,\infty}}
		\\&\lesssim \norm {u_0}_{  \dot B^{-\frac{d}{2} }_{2,\infty}} + \norm {u\otimes u}_{\widetilde L^2_t \dot B^{-\frac{d}{2}}_{2,\infty}}  + \norm {J\times B}_{L^1_{t,x}} 
		\\&\lesssim \norm {u_0}_{  \dot B^{-\frac{d}{2} }_{2,\infty}} + \norm {u }_{ L^2_t \dot B^{\frac{d}{2}}_{2,1}}\norm {u }_{  L^\infty _t \dot B^{-\frac{d}{2}}_{2,\infty}}  + \norm E
_{L^2_{t,x}} \norm B 
_{L^2_{t,x}} ,
	\end{aligned} 
\end{equation}
where we employed the product laws from Lemma \ref{lemma:product}, once again (note that the appearance of the  Chemin--Lerner space $ \widetilde L^2_t \dot B^{-\nicefrac{d}{2}+ 1}_{2,\infty} $ here is complementary as it does not serve at all in controlling the terms from the right-hand side. We refer to \cite{bcd11} for the precise definition of these spaces).    
Except for the bound  of  $ B $ in $
{L^2_{t,x}} $,   all the other norms appearing on the right-hand side above are finite and under control globally in time. We are thus only left with  the details  to control  $ B $ in $
{L^2_{t,x}} $, which we now handle separately.  To that end, recasting the equation of $B$ as a perturbation of the heat equation, we find that 
 \begin{equation*}
 	  \partial_t B _j- \Delta B_j  = \Delta_j \nabla \times (u\times B) - c^{-2}\partial_{tt} B_j ,
 \end{equation*}
 where  $B_j\bydef \Delta_j B$ is the localized magnetic field via Littlewood--Paley dyadic blocs $(\Delta_j)_{j\in \mathbb Z}$. Thus, it follows, by an $L^2$ energy estimate, that 
 \begin{equation}\label{parabolic_inequality}
 	\norm {B}_{L^\infty_t \dot H^{-1}\cap L^2_{t,x}} ^2 \lesssim \norm {B_0}_{\dot H^{-1}} ^2+  \norm { u\times B}_{L^2_t \dot H^{-1}}^2 - c^{-2} \sum_{j\in \mathbb Z} 2^{-2j} \int_0^t \int_{\mathbb R^d}  \partial_{tt} B_j  \cdot B_j.
 \end{equation}
 By an integration by parts in time and H\"older inequalities, we write 
 \begin{equation*}
 	\begin{aligned}
 		\left|  \sum_ {j\in \mathbb Z} 2^{-2j}	\int_0^t \int_{\mathbb R^d} \partial_{tt} B_j  \cdot B_j \right| 
 		&\leq  \norm {\partial_t B}_{L^2_t \dot H^{-1}}^2 + \norm {\partial_t B}_{L^\infty_t \dot H^{-1}} \norm { B}_{L^\infty_t \dot H^{-1}} + \norm {B_0}_{\dot H^{-1}}\norm {E_0}_{L^2} 
 		\\
 		&\leq \norm {cE}_{L^2_{t,x}}^2 + c\norm {E}_{L^\infty_tL^2} \norm { B}_{L^\infty_t \dot H^{-1}} + c\norm {B_0}_{\dot H^{-1}}\norm {E_0}_{L^2} ,
 	\end{aligned}
 \end{equation*}
 where we used Faraday's equation  from \eqref{Main_System}.
 Using Young inequalities, once again, and incorporating the preceding estimate into the parabolic inequality \eqref{parabolic_inequality}  yields that 
 \begin{equation*}
 	\norm {B}_{L^\infty_t \dot H^{-1}\cap L^2_{t,x}} ^2 \lesssim \norm {B_0}_{\dot H^{-1}} ^2 + \mathcal E_0^2\Big (1+ \norm {(u_0,E_0,B_0)}_{\dot H^{\frac{d}{2}-1}} \Big )  +  \norm { u }_{L^2_t \dot B^{\frac{d}{2}}_{2,1}}^2 \norm {   B}_{L^\infty _t \dot H^{-1}}^2  ,
 \end{equation*}
 where we also used that 
 \begin{equation*}
 	\begin{aligned}
 		\norm {cE}_{L^2_{t,x}}^2 
 		&\leq \norm {j}_{L^2_{t,x}}^2 + \norm {u\times B}_{L^2_{t,x}}^2
 		\\
 		&\leq \mathcal E_0^2 + \norm {u }_{L^2_{t}L^\infty}^2\norm { B}_{L^\infty_t L^2}^2
 		\\
 		&\leq \mathcal E_0^2 \Big (1+ \norm {(u_0,E_0,B_0)}_{\dot H^{\frac{d}{2}-1}} \Big ) .
 	\end{aligned}
 \end{equation*}
 In fact, we emphasize that one can establish, along the same lines of proof as above, the slightly stronger estimate 
 \begin{equation*}
 	\begin{aligned}
 		&\norm {B}_{L^\infty (t_0,t; \dot H^{-1})\cap L^2 (t_0,t; L^2)} ^2 
 		\\
 		&\qquad\lesssim \norm {B (t_0)}_{\dot H^{-1}} ^2 + \mathcal E_0^2\Big (1+ \norm {(u,E,B)(t_0)}_{\dot H^{\frac{d}{2}-1}} \Big )  +  \norm { u }_{L^2(t_0,t; \dot B^{\frac{d}{2}}_{2,1})}^2 \norm {   B}_{L^\infty (t_0,t; \dot H^{-1})}^2  ,
 	\end{aligned}
 \end{equation*}
 for all $0\leq t_0 \leq t$. 
 Hence, by a bootstrap argument in time (see Lemma \ref{lemma:bootstrap}), we finally deduce that   
 	\begin{equation*}
 	\norm {B}_{L^\infty_t \dot H^{-1}\cap L^2_{t,x}} ^2 \leq C \left( \norm {B_0}_{\dot H^{-1}} ^2 + \mathcal E_0^2 \Big (1+ \norm {(u_0,E_0,B_0)}_{\dot H^{\frac{d}{2}-1}} \Big ) \right)   \exp\left( C  \norm { u }_{L^2_t \dot B^{\frac{d}{2}}_{2,1}}^2  \right),
 \end{equation*}
 for some constant $C>0$. By virtue of the control of $u$ from Theorem \ref{thm:critical:wp}, we end up with the bound \begin{equation*}
 	\norm {B}_{L^\infty_t \dot H^{-1}\cap L^2_{t,x}} \leq C_0,
 \end{equation*}
 where $C_0$ depends only on the initial data.
 
 Turning back to \eqref{negative:es:u}, we obtain by employing the preceding bound for $B$, that 
 \begin{equation*}
 	\norm {u}_{L^\infty_t \dot  B^{-\frac{d}{2}}_{2,\infty} \cap \widetilde L^2_t \dot B^{-\frac{d}{2}+ 1}_{2,\infty}} 
		 \lesssim \norm {u_0}_{  \dot B^{-\frac{d}{2} }_{2,\infty}}  + C_0   + \norm {u }_{ L^2_t \dot B^{\frac{d}{2}}_{2,1}}\norm {u }_{  L^\infty _t \dot B^{-\frac{d}{2}}_{2,\infty}}  
 \end{equation*}
 As before, we emphasize that one can in fact prove the slightly stronger estimate 
 \begin{equation*}
 	\norm {u}_{L^\infty(t_0,t; \dot  B^{-\frac{d}{2}}_{2,\infty}) \cap \widetilde L^2(t_0,t; \dot B^{-\frac{d}{2}+ 1}_{2,\infty})} 
		 \lesssim \norm {u (t_0)}_{  \dot B^{-\frac{d}{2} }_{2,\infty}}  + C_0   + \norm {u }_{ L^2(t_0,t; \dot B^{\frac{d}{2}}_{2,1})}\norm {u }_{  L^\infty (t_0,t; \dot B^{-\frac{d}{2}}_{2,\infty})}  ,
 \end{equation*}
 for all $0\leq t_0 \leq t$.
 Once again, by a bootstrap argument in time (Lemma \ref{lemma:bootstrap}), we finally arrive at   
 \begin{equation*}
 	\norm {u}_{L^\infty_t \dot  B^{-\frac{d}{2}}_{2,\infty} \cap \widetilde L^2_t \dot B^{-\frac{d}{2}+ 1}_{2,\infty}} \lesssim \left( \norm {u_0}_{  \dot B^{-\frac{d}{2} }_{2,\infty}}  + C_0 \right) \exp\left( C  \norm { u }_{L^2_t \dot B^{\frac{d}{2}}_{2,1}}^2  \right),
 \end{equation*}
 for some constant $C>0$ and up to a suitable change in the constant $C_0$. Finally, by the global bounds previously established in Theorem \ref{thm:critical:wp}, the claimed global control of the velocity field in $L^\infty_t \dot  B^{-\nicefrac{d}{2}}_{2,\infty} \cap \widetilde L^2_t \dot B^{-\nicefrac{d}{2}+ 1}_{2,\infty}$ follows.
 
 We now turn our attention to the estimate of the electromagnetic field. To that end, writing the localized energy estimate for the Maxwell equation
 \begin{equation*}
 	\norm {(B_j,E_j)(t)}_{L^2}^2 + c^2 \int_0^t \norm {E_j(\tau)}_{L^2}^2 d\tau \lesssim  \norm {\Delta_j(B_0,E_0)}_{L^2}^2  + \int_0^t \norm {\Delta_j (u\times B)}_{L^2}^2,
 \end{equation*} 
 yields, by means of product law estimates, that 
 \begin{equation*}
 	\begin{aligned}
 		\norm {(E,B)}_{L^\infty_t B^{-\frac{d}{2}}_{2,\infty}} ^2
 		& \lesssim  \norm { (B_0,E_0)}_{B^{-\frac{d}{2}}_{2,\infty}}^2  +   \norm { u\times B}_{ L^2_t B^{-\frac{d}{2}}_{2,\infty}}^2
 		\\
 		& \lesssim  \norm { (B_0,E_0)}_{B^{-\frac{d}{2}}_{2,\infty}}^2  +  \int_0^t  \norm { u (\tau)}_{  B^{\frac{d}{2}}_{2,1}}^2 \norm {   B  (\tau) }_{   B^{-\frac{d}{2}}_{2,\infty}}^2 d\tau.
 	\end{aligned}
 \end{equation*}
 Applying Gron\"wall inequality, we arrive at a global bound for  $(E,B)$ in $L^\infty_t B^{-\frac{d}{2}}_{2,\infty}$. More precisely, we find that  \begin{equation*}
 	\norm {(E,B)}_{L^\infty_t B^{-\frac{d}{2}}_{2,\infty}}  \lesssim  \norm { (B_0,E_0)}_{B^{-\frac{d}{2}}_{2,\infty}} \exp\left( C  \norm { u }_{L^2_t \dot B^{\frac{d}{2}}_{2,1}}^2  \right) \lesssim   C_0,
 \end{equation*}
 after a suitable change in the definition of the constant $C_0$. This concludes the proof of the proposition.
\end{proof}

  \section{Decay estimates }

  In this section, we investigate the long-time behavior of the global solution constructed in Theorem \ref{thm:critical:wp}. More precisely, we prove Theorem \ref{thm:decay} by establishing an optimal decay rate for the solution   under the additional assumption
  \begin{equation*}
  	u_0\in \dot{B}_{2,\infty}^{-d/2}(\mathbb{R}^d).
  \end{equation*}
Although this assumption is weaker than the condition $u_0\in L^1(\mathbb{R}^d)$, in view of the embedding
\begin{equation*}
	L^1(\mathbb{R}^d)\hookrightarrow \dot B_{2,\infty}^{- \frac{d}{2}}(\mathbb{R}^d),
\end{equation*}
it still allows us to recover the optimal decay rate associated with the heat kernel, namely,
\begin{equation*}
	\norm{  e^{t\Delta}u_0}_{ \dot H^s(\mathbb{R}^d)}
\lesssim (1+t)^{- \frac{s}{2} - \frac{d}{2}}.
\end{equation*}
The proof of Theorem \ref{thm:decay} is based on two main ingredients: First, we establish time-differential inequality of a Lyapunov functional     that is poitwisely equivalent to the Sobolev norm $ \dot H^s$. Then, by adapting an argument inspired from the Fourier splitting method, we show that this Lyapunov functional enjoys the above optimal decay estimate.

   \begin{proof}[Proof of Theorem \ref{thm:decay}]
   Given the boundedness of $(u,E,B)$ in $\dot B^{-\frac{d}{2}}_{2,\infty}$, it is sufficient to prove the decay estimate for $s=\frac{d}{2}-1$, for the remaining values of $s\in (-\frac{d}{2},\frac{d}{2}-1)$
    would follow eventually by a simple interpolation argument.
    
  Now, observe that \eqref{Model} can be recast in the Fourier side as
   \begin{equation}\label{Main_Fourier}
\left\{\
\begin{aligned}
&\partial_t \widehat u+ |\xi|^2\widehat u+i\xi \widehat p
   =\widehat{j\times B}-\widehat{u\cdot\nabla u},\\
&\frac{1}{c} \partial_t \widehat E-i\xi\times \widehat B=-\widehat j,\\
&\frac{1}{c}\partial_t \widehat B+i\xi\times \widehat E=0,\\
&\xi\cdot \widehat u=\xi\cdot \widehat B=0,\\
&\widehat j=c\widehat E+\widehat{u\times B}.
\end{aligned}
\right.
\end{equation}
Taking the real part of the scalar product in $\mathbb{C}^d $ of the first three equations with   $ (\widehat u, \widehat E, \widehat B)$, respectively, yields that\begin{equation}\label{Fourier_Energy}
\begin{aligned}
\frac12\frac{d}{dt} |\widehat U|^2
+  |\xi|^2|\widehat u|^2
+|c\widehat E|^2
=I \bydef \sum_{\ell =1}^4 I_\ell  ,
\end{aligned}
\end{equation}
where
\begin{equation*}
  	U\bydef (u,E,B), 
  \end{equation*}  
and
\begin{equation*}
	\begin{aligned}
		I_1&=c
\operatorname{Re}
\left \langle
\widehat{E\times B},  \widehat u
\right \rangle,
\qquad I_2=
\operatorname{Re}
\left \langle
\widehat{(u\times B)\times B},
 \widehat u
\right \rangle,
\\
I_3&=-c
\operatorname{Re}
\left \langle
\widehat{u\times B},
 \widehat E
\right \rangle
, 
\qquad I_4=-
\operatorname{Re}
\left \langle
\widehat{u\cdot\nabla u},
\widehat u
\right \rangle.
	\end{aligned}
\end{equation*} 

It follows from the energy identity \eqref{Fourier_Energy} that, at first glance, only the components $\widehat{u}$ and $\widehat{E}$ of $\widehat{U}$ are directly dissipated. Nevertheless, owing to the coupling structure of the system \eqref{Main_Fourier}, the dissipation of these two components can be transferred to the $\widehat{B}$ component as well. To capture this hidden dissipation mechanism, we introduce the following auxiliary quantity
\begin{equation}\label{Functional_F}
F_c(t,\xi)
\bydef 
-\frac{1}{c} \mathrm{Re}\,\langle i\xi\times \widehat{E},\widehat{B}\rangle ,
\end{equation}
which will be the second part in the definition of the Lyapunov functional of our interest. 
Applying $(i\xi \times \cdot )$ to the second equation in \eqref{Main_Fourier} yields  
\begin{equation}\label{curl_2}
\frac{1}{c}\partial_t (i\xi\times \widehat{E})-i\xi\times(i\xi\times \widehat{B})=- i\xi\times\widehat{j}.
\end{equation}
By further taking the dot product with $-\widehat{B}$ we obtain  
\begin{eqnarray*}
-\frac{1}{c}\frac{d}{dt}\langle  (i\xi\times \widehat{E}),\widehat{B}\rangle+\frac{1}{c} \langle  (i\xi\times \widehat{E}),\partial_t\widehat{B}\rangle+\langle i\xi\times(i\xi\times \widehat{B}), \widehat{B}\rangle=c\langle i\xi\times\widehat{E},\widehat{B} \rangle+\langle i\xi\times\widehat{u\times B},\widehat{B} \rangle. 
\end{eqnarray*}
Note that, as a result of the identity $\xi\cdot \widehat{B}=0$, one has    $\xi\times(\xi\times \widehat{B})=-|\xi|^2\widehat{B}$. Therefore, combined with the third equation from \eqref{Main_Fourier}, the previous $\xi$-dependent ODE implies that 
\begin{eqnarray}\label{F_dt}
\frac{d}{dt}F_c(t,\xi)+|\xi|^2|\widehat{B}|^2=|\xi\times \widehat{E}|^2+c\mathrm{Re}\langle i\xi\times\widehat{E},\widehat{B} \rangle +\mathrm{Re}\langle i\xi\times\widehat{u\times B},\widehat{B} \rangle.
\end{eqnarray}
Thus, applying Cauchy-Schwarz inequality together with Young's inequality, implies, for any $\epsilon>0$,  that
\begin{equation*}
c|\mathrm{Re}\langle i\xi\times\widehat{E},\widehat{B} \rangle|+\mathrm{Re}\langle i\xi\times\widehat{u\times B},\widehat{B} \rangle\leq \epsilon |\xi|^2|\widehat{B}|^2+k(\epsilon)\left(|\widehat{cE}|^2+|\widehat{u\times B}|^2\right),
\end{equation*}
for some positive constant $k(\epsilon)>0$.
Plugging this     into \eqref{F_dt} leads (for $c\geq 1$) to the differential inequality 
\begin{equation}\label{B_dissipation_nonlinear}
\frac{d}{dt}F_c(t,\xi)
+
(1-\epsilon)|\xi|^2|\widehat B|^2
\leq
k(\epsilon)(1+|\xi|^2)|c\widehat E|^2
+
k(\epsilon)|\widehat{u\times B}|^2 .
\end{equation}
Now, we define the Lyapunov functional 
\begin{equation}\label{Lyapunov_F}
L_c(t,\xi)\bydef \frac{\gamma}{2} (1+|\xi|^2)|\widehat U(t,\xi)|^2+F_c(t,\xi),  
\end{equation}
where $\gamma$ is a large positive constants which will be determined, later on. 
Hence, combining  \eqref{Fourier_Energy} and \eqref{B_dissipation_nonlinear}, we obtain that 
\begin{equation}\label{L_dt_k}
\begin{aligned}
&\frac{d}{dt}L_c(t,\xi)+  \gamma(1+|\xi|^2)|\xi|^2| \widehat u|^2 +(\gamma-k(\epsilon))(1+|\xi|^2)|c\widehat E|^2+(1-\epsilon)|\xi|^2|\widehat B|^2\\
& \qquad \qquad \qquad\leq  k(\epsilon)|\widehat{u\times B}|^2+ (1+|\xi|^2)I(t,\xi ).
\end{aligned}
\end{equation}
Next, choosing $\epsilon$ small enough, at first, then making a choice for $\gamma$ large enough, we end up with the differential inequality  
\begin{equation}\label{L_dt_k_2}
\begin{aligned}
&\frac{d}{dt}L_c(t,\xi)+k_0\left((1+|\xi|^2)|\xi|^2|  \widehat u|^2
+(1+|\xi|^2)|c\widehat E|^2+|\xi|^2|\widehat B|^2\right) \\
& \qquad \qquad \qquad\lesssim |\widehat{u\times B}|^2+ (1+|\xi|^2)I(t,\xi ),
\end{aligned}
\end{equation}
for some universal constant $k_0>0$.
Hence, it follows that  
\begin{equation*} 
\begin{aligned}
&\frac{d}{dt}L_c(t,\xi)+\frac{k_0}{2} |\xi|^2|\widehat  U(t,\xi)|^2+\frac{k_0}{2}(1+|\xi|^2) ( |c\widehat E|^2 + |\xi|^2 |\widehat u|^2)\\
&\lesssim\, |\widehat{u\times B}|^2+ (1+|\xi|^2)I(t,\xi ). 
\end{aligned}
\end{equation*} 
Notice in passing that,   for $\gamma$ being taken large enough, it holds, for any $(t,\xi)\in \mathbb R^+\times \mathbb R^d$, that 
\begin{equation}\label{Equiv_E_k_L_k}
L_c(t,\xi)\sim (1+|\xi|^2) |\widehat U(t,\xi)|^2, 
\end{equation}
where the implicit constant is universal and does not depend on the speed of light, too.
We now take care of the source term $I(t,\xi)$. To that end, we employ Young inequalities  to estimate each summand in its definition. Accordingly,  we find that   
\begin{equation*}
	\begin{aligned}
		I &\leq  \frac{k_0}{4} \left(    |c \widehat E|^2 + |\xi|^2 |\widehat u|^2   \right)+ \frac{C(k_0)}{|\xi|^2}   \left( |c\widehat { E\times B} | ^2  + | \widehat {( u\times B) \times B}| ^2 + | \widehat { u\cdot \nabla u}| ^2 \right) 
		\\
		&\qquad + C(k_0) |\widehat { u\times B}|^2,
	\end{aligned}
\end{equation*}
for some constant $C(k_0)>0$.
Therefore, we end up with the $\xi$-dependent differential inequality   
\begin{equation}\label{L_dt_k_2}
\begin{aligned}
\frac{d}{dt}  \left( \frac{L_c(t,\xi)}{ 1+ |\xi|^2 } \right)   + b_0      \left( \frac{|\xi|^2}{ 1+ |\xi|^2 } \right)  \left( \frac{ L_c(t,\xi)}{1+|\xi|^2}\right)  
\leq  J(t,\xi),
\end{aligned}
\end{equation} 
for some $b_0>0$, where 
\begin{equation*}
	J(t,\xi) \bydef \frac{C(k_0)}{|\xi|^2}   \left( |c\widehat { E\times B} | ^2  + | \widehat {( u\times B) \times B}| ^2 + | \widehat { u\cdot \nabla u}| ^2 \right) + C(k_0) |\widehat { u\times B}|^2.
\end{equation*}
Introducing the time-dependent functions 
\begin{equation*}
	\mathcal{L}_c(t) \bydef \int_{\mathbb R^d} \left ( \frac{L_c(t,\xi)}{ 1+ |\xi|^2 } \right) |\xi|^{d-2}d\xi , \qquad\mathcal{D}_c (t) \bydef \int_{\mathbb R^d}  \left( \frac{|\xi|^2}{ 1+ |\xi|^2 } \right)  \left( \frac{ L_c(t,\xi)}{1+|\xi|^2}\right)  |\xi|^{d-2}d\xi ,
\end{equation*}
we find, upon integrating \eqref{L_dt_k_2}  against the measure $|\xi|^{d-2}d\xi $, that 
\begin{equation*}
 \begin{aligned}
 	\frac{d}{dt} \mathcal{L}_c(t) + b_0\mathcal{D}_c(t) 
 	&\leq   \int_{\mathbb R^d} J(t,\xi) |\xi|^{d-2}d\xi
 	\\
 	& \lesssim \norm {cE\times B}_{\dot H^{\frac{d}{2}-2}}^2 + \norm {( u\times B) \times B}_{\dot H^{\frac{d}{2}-2}}^2
 	\\
 	& \quad  + \norm { u\cdot \nabla u}_{\dot H^{\frac{d}{2}-2}}^2 + \norm { u\times B }_{\dot H^{\frac{d}{2}-1}}^2.
 \end{aligned}
\end{equation*}
Thus, by employing the product laws 
\begin{equation*}
	\dot H^{\frac{d}{2}-1} \cdot\dot H^{\frac{d}{2}-1}\hookrightarrow \dot H^{\frac{d}{2}-2}(\mathbb R^d) \qquad \text{ and } \qquad \dot H^{\frac{d}{2}-1} \cdot \dot B^{\frac{d}{2}}_{2,1} \hookrightarrow \dot H^{\frac{d}{2}-1} (\mathbb R^d),
\end{equation*}
together with the equivalence relation \eqref{Equiv_E_k_L_k}, we obtain 
\begin{equation*}
 \begin{aligned}
 	\frac{d}{dt} \mathcal{L}_c(t) + b_0\mathcal{D}_c(t) 
 	&\lesssim a(t) \norm {(u,E,B)(t)}_{\dot H^{\frac{d}{2}-1}}^2 \sim a(t)  \mathcal{L}_c(t),
 \end{aligned}
\end{equation*}
where, by Theorem \ref{thm:critical:wp},
\begin{equation*}
	t\mapsto a(t) \bydef \norm {cE(t)}_{\dot H^{\frac{d}{2}-1}}^2+ \norm { u(t)}_{\dot B^{\frac{d}{2}}_{2,1}}^2 \left (1+ \norm {B(t)}_{\dot H^{\frac{d}{2}-1}}^2 \right) \in L^1(\mathbb R^+).
\end{equation*} 

We are now in a position to adapt the Fourier splitting method to the last differential inequality above, and conclude the proof of the decay rate. To that end, we write, for 
\begin{equation*}
	\lambda (t) \bydef \left( \frac{2\eta }{b_0 (1+t)} \right)^\frac{1}{2} \qquad \text{ and } \qquad \eta >d-1,
\end{equation*}
that
\begin{equation*}
	\begin{aligned}
		\mathcal{D}_c (t)
		 &\geq    \int_{\{  |\xi| >\lambda (t) \}} \left ( \frac{L_c(t,\xi)}{ 1+ |\xi|^2 } \right) \left( 
		 \frac{|\xi|^2}{1+|\xi|^2} \right) |\xi|^{d-2}d\xi
		 \\
		&\geq    \frac{ \lambda^2(t) }{1+  \lambda^2(t) } \int_{\{  |\xi| >\lambda (t) \}} \left ( \frac{L_c(t,\xi)}{ 1+ |\xi|^2 } \right) |\xi|^{d-2}d\xi
		\\
		&=       \frac{ \lambda^2(t) }{1+  \lambda^2(t) } \mathcal{L}_c(t)-   \frac{ \lambda^2(t) }{1+  \lambda^2(t) }  \int_{\{  |\xi| \leq \lambda (t) \}} \left ( \frac{L_c(t,\xi)}{ 1+ |\xi|^2 } \right) |\xi|^{d-2}d\xi
		\\
		&\geq        \frac{ \lambda^2(t) }{1+  \lambda^2(t) }  \mathcal{L}_c(t)- C    \lambda^ { d}(t)    \int_{\{  |\xi| \leq \lambda (t) \}} |\widehat U(t)|^2d\xi
		\\
		&\geq       \frac{ \lambda^2(t) }{1+  \lambda^2(t) }  \mathcal{L}_c(t)-  C   \lambda^ { 2d}(t)   \norm { U(t)}_{\dot B_{2,\infty}^{-\frac{d}{2}}}^2,
	\end{aligned}
\end{equation*}
where we used the equivalence \eqref{Equiv_E_k_L_k}.
Therefore, for $\lambda(t)\leq   1$ (which is realized for sufficiently large $t$),
we obtain the lower bound
\begin{equation*}
	\mathcal D_c(t) \geq \frac{\lambda ^2(t)}{2} \mathcal L_c(t) - C   \lambda^ { 2d}(t)   \norm { U(t)}_{\dot B_{2,\infty}^{-\frac{d}{2}}}^2.
\end{equation*}
 By further employing Proposition \ref{Prop_negative_regularity} to control  $U$ in $L^\infty(\mathbb R^+; \dot B^{-\frac{d}{2}}_{2,\infty}(\mathbb R^d))$, we deduce that 
\begin{equation*}
	\frac{d}{dt} \mathcal{L}_c(t) + \frac{\eta }{ 1+t} \mathcal{L}_c(t) \leq a(t)  \mathcal{L}_c(t)  +C \left( 1+t \right) ^ {  -d  } ,
\end{equation*}
where $C>0$ depends on the initial data, only.
Multiplying by $(1+t)^\eta$   
\begin{equation*}
	\frac{d}{dt}\Big ( (1+t)^\eta \mathcal{L}_c(t)\Big ) \leq a(t) (1+t)^\eta \mathcal{L}_c(t)  +C \left( 1+t \right) ^ {\eta  -d   },
\end{equation*}
and applying Gronuwall's inequality, we find that 
\begin{equation*}
	(1+t)^\eta \mathcal{L}_c(t) \leq  \left(  \mathcal{L}_c(0)   + C (1+t)^ {\eta - d+1} 
	\right) \exp \left ( \int_{0}^\infty a(\tau) d\tau  \right).
\end{equation*} 
Once again, making use of  the equivalence relation \eqref{Equiv_E_k_L_k}, which in particular gives that 
\begin{equation*}
	\mathcal{L}_c(t) \sim \norm {U(t)}_{\dot H^{\frac{d}{2}-1}}^2,
\end{equation*}
we finally arrive at the  estimate for the decay rate   
\begin{equation*}
	  \norm {U(t)}_{\dot H^ {\frac{d}{2}-1}}^2  \lesssim   (1+t)^ {-\eta } +  (1+t)^ {-d +1 } \lesssim  (1+t)^ {-d +1 },
\end{equation*} 
thereby concluding the proof. 
   \end{proof}

\appendix

\section{On a Gr\"onwall-type inequality} 
We conclude by outlining the proof of a Gr\"onwall-type estimate for the solutions of an integral inequality of the form
\begin{equation}\label{GR:setup}
	f'(t) \leq bf(t_0) + \big(a(t) -a(t_0)\big) f(t), \quad \text{ for all } 0\leq t_0 \leq t,
\end{equation}
under suitable assumptions on the coefficient function $a$, where $b\geq 1$ is a universal constant.
\begin{lemma}\label{lemma:bootstrap} Let $T\in \mathbb R^+\cup \{ \infty\}$ and  assume that $a\in C_b([0,T); \mathbb R^+)$ is non-decreasing on the time interval $[0,T)$, with the convention 
\begin{equation*}
	a(T)\bydef \lim_{t\to T} a(t).
\end{equation*}
	Then, any solution of \eqref{GR:setup} on the time interval $[0,T)$ enjoys the bound
	\begin{equation*}
		f(t) \leq f(0) \exp \left( C_b   a(T)\right), \quad \text{ for all } t \in [0,T),
	\end{equation*}
	for some suitable universal constant $C_b>0$ depending only on $b$.
\end{lemma}
\begin{proof}
	Throughout the  proof, which  is based  on a bootstrap argument, we may assume that $T<\infty$. To proceed, we introduce the sequence 
	\begin{equation*}
		t_{n+1} \bydef \sup \left \{t\in [t_{n}, T):a(t) -a(t_{n}) < \frac{1}{2}\right \},
	\end{equation*}
	for any integer  $n\geq 0$, where $t_0=0$. Notice first that, by continuity of $a$, if $t_{n+1}<T$,   one necessarily has  
	\begin{equation*}
		a(t_{n+1})- a(t_{n})=\frac{1}{2}.
	\end{equation*}
	Moreover, as $T<\infty$, there exists $N\in \mathbb N$ such that $ t_N=T$, i.e., the subdivision of the interval 
	\begin{equation*}
		[0,T)= \bigcup_{n\in \mathbb N} [t_n, t_{n+1})
	\end{equation*}
	must be finite 
	\begin{equation}\label{subsdivision}
		[0,T)= \bigcup_{0\leq n\leq N-1} [t_n, t_{n+1}).
	\end{equation}	 
	Indeed, if the sequence $(t_n)_{n\in \mathbb N}$ were to be infinite, then one has, for any $n\in \mathbb N$, that  
	\begin{equation*}
		\frac{1}{2}= a(t_{n+1})- a(t_{n}),
	\end{equation*}
	thereby 
	\begin{equation*}
		\frac{k}{2}= \sum_{n=0}^{k-1} \left( a(t_{n+1})- a(t_{n})\right) = a(t_k) - a(0),
	\end{equation*}
	which entails that $$\lim_{t\to T}a(t)=\infty,$$
	thereby contradicting the assumption that $a\in C_b([0,T);\mathbb R^+)$.
	Thus, we deduce that the finite subdivision \eqref{subsdivision} holds, and that 
	\begin{equation*}
		a(T)- a(t_{N-1})\leq \frac{1}{2} , \qquad a(t_{n+1})- a(t_{n})=\frac{1}{2} ,
	\end{equation*}
	for all integer $0\leq n\leq N-2$. It is then readily seen that 
	\begin{equation*}
		\frac{N-1}{2}= a(T_{N-1})- a(0) \leq a(T) - a(0),
	\end{equation*} 
	which in turn provides the bound (for $N\geq 2$)
	\begin{equation}\label{N:bound}
		N \leq 4 a(T)   .
	\end{equation}
	Now, we estimate \eqref{GR:setup} on each subinterval $[t_n,t_{n+1}]$ to obtain 
	\begin{equation*}
		f(t) \leq b f(t_{n}) + \frac{1}{2} f(t), \quad \text{ for all } t\in [t_{n}, t_{n+1}],
	\end{equation*}
	whereby 
	\begin{equation*}
		f(t) \leq 2b f(t_{n}), \quad \text{ for all } t\in [t_{n}, t_{n+1}],
	\end{equation*}
	for all integers $0\leq n \leq N-1$. By induction, this leads to the bound 
	\begin{equation*}
		f(t) \leq (2b)^{N} f(0) , 
	\end{equation*}
	whence, by \eqref{N:bound}, we conclude that  
	\begin{equation*}
		f(t) \leq  e^{4 \log (2b) a(T)} f(0), 
	\end{equation*}
	 for all $t\in [0,T).$ This completes the proof of the lemma.
\end{proof}

\section*{acknowledgement} 

H. Houamed acknowledges support from the  BAGEP Award of the Science Academy, T\"urkiye.

\bibliographystyle{plain}
\bibliography{plasma.bib}

\end{document}